\documentclass[12pt]{amsart}
\usepackage{amsmath,amsfonts,amssymb,amsthm}
\usepackage{anysize}
\marginsize{2cm}{2cm}{2,1cm}{2,1cm}
\usepackage{graphicx}
\usepackage{color}
\usepackage{multirow}

\usepackage{ textcomp }
\usepackage{schemata}

\usepackage{enumerate}

\def\z{\mathfrak{z}}
\def\u{\mathfrak{u}}

\def\g{\mathfrak{g}}
\def\h{\mathfrak{h}}
\def\n{\mathfrak{n}}
\def\v{\mathfrak{v}}

\def\C{\mathbb{C}}
\def\R{\mathbb{R}}

\def\ad{\operatorname{ad}}
\def\dim{\operatorname{dim}}

\def\alt{\raise1pt\hbox{$\bigwedge$}}
\def\pint{\langle \cdotp,\cdotp \rangle }
\def\la{\langle}
\def\ra{\rangle}

\newcommand\ggo{{\mathfrak g}}

\newcommand{\iprod}{\mathbin{\lrcorner}}

\def\Im{\operatorname{Im}}

\theoremstyle{plain}
\newtheorem{teo}{\bf Theorem}[section]
\newtheorem{cor}[teo]{\bf Corollary}
\newtheorem{prop}[teo]{\bf Proposition}
\newtheorem{lema}[teo]{\bf Lemma}

\theoremstyle{definition}

\newtheorem{ejemplo}[teo]{\bf Example}

\theoremstyle{remark}
\newtheorem{rem}[teo]{\bf Remark}

\newcommand\aff{\mathfrak{aff}}

\title{A survey on invariant conformal killing forms on Lie groups}
\author{A. Herrera}
\address{FCEFyN, Universidad Nacional de C\'{o}rdoba, Ciudad Universitaria, 5000 C\'{o}rdoba, Argentina}
\email{cecilia.herrera@unc.edu.ar}
\author{M. Origlia}
\address{FaMAF-CIEM, Universidad Nacional de C\'{o}rdoba, Ciudad Universitaria, 5000 C\'{o}rdoba, Argentina}
\email{marcos.origlia@unc.edu.ar}

\date{}

\begin{document}

\begin{abstract}
In this survey we review recent results on left-invariant conformal Killing $p$-forms on Lie groups endowed with a left-invariant metric. We also mention interesting open questions that could lead into future research.
\end{abstract}

\maketitle
\tableofcontents

\section{Introduction}

Killing forms appeared as a generalization of Killing vector fields; they were introduced by  Yano in 1951 (see \cite{Yano1}). Yano considered a $p$-form  defined on a Riemannian manifold $(M,g)$ and extended the notion of Killing vector field to this case, that is, a $p$-form  $\eta$  is Killing  if it satisfies the following equation
\begin{equation}\label{yanoequationoriginal}
\nabla\eta(X_1,X_2,\dots,X_{p+1})+\nabla \eta(X_2,X_1,\dots,X_{p+1})=0, 
\end{equation}
for all vector fields $X_i$, \color{black}where $\nabla$ denotes the Levi-Civita connection associated to the metric $g$.  
In 1968, Tachibana renamed \eqref{yanoequationoriginal} in  \cite{Tachibana1} as the Killing-Yano equation. One year later in \cite{Tachibana},  the same author extended the concept of Killing $2$-forms to conformal Killing $2$-forms. In the same year, Kashiwada continued this generalization and defined conformal Killing (or conformal Killing-Yano) $p$-forms for  $p\geq 2$ (\cite{Kashiwada}). A $p$-form $\eta$  is called conformal Killing  if it exists a $(p-1)$-form $\theta$ such that the following equation is satisfied
\begin{align}\nonumber\label{Tachibanaequationoriginal}
		\nabla \eta(X_1,X_2,\dots,X_{p+1})+\nabla& \eta(X_2,X_1,\dots,X_{p+1})=2g(X_1, X_2) \theta\left(X_3 , \dots , X_{p+1}\right)\\
		&  -\displaystyle{\sum_{i=3}^{p+1}}(-1)^i g\left(X_1,X_i\right)\theta\left(X_1, X_3 ,\dots , \widehat{X_i} ,\dots, X _{p+1} \right) \\
		& -\displaystyle{\sum_{i=3}^{p+1}}(-1)^i g\left(X_2,X_i\right)\theta\left(X_2, X_3 ,\dots, \widehat{X_i}, \dots , X _{p+1}\right) \nonumber
\end{align} 
for $X_1,X_2,\dots,X_{p+1}$ vector fields and where $\widehat{X_i}$ means that $X_i$ is omitted. We will denote conformal Killing or conformal Killing-Yano forms by CKY for short, similarly KY will mean Killing or Killing-Yano. We will call \textit{strict} CKY to CKY forms which are not KY.

An intuitive example of Killing $2$-forms are given by nearly K\"ahler manifolds $(M,g,J)$ whose fundamental $2$-form $\omega$ given by $\omega(X,Y)=g(JX,Y)$ is Killing. Similarly, 
the canonical $2$-form of a Sasakian manifold is a strict CKY $2$-form (\cite{Semmelmann}).

The matematical development of these forms has  taken interest in the last 25 years, because they are considered a powerful tool in the general relativity and supersymmetric quantum field theory. We recommend to see \cite{O.Santillan} for more details.

In 2001, S. E. Stepanov motivated by the relationship between the Maxwell equations of relativistic electrodynamics and conformal Killing 2-forms, studied the geometry of the space of conformal Killing  $p$-forms in \cite{Stepanov}. Then in 2003, U. Semmelman in \cite{Semmelmann} introduced a different point of
view, he described a conformal Killing $p$-form as a form in the kernel of  a first order elliptic differential operator. 
Equivalently, a $p$-form $\eta$ is  conformal Killing  on  a $n$-dimensional Riemannian manifold $(M,g)$ if it satisfies 
for any vector field $X$ the following equation
\begin{equation}\label{ckyManifolds}
	\nabla_X  \eta=\frac{1}{p+1}X \lrcorner\mathrm{d}\eta-\frac{1}{n-p+1}X^*\wedge \mathrm{d}^*\eta,
\end{equation}
where $X^*$ is the dual 1-form of $X$,  $\mathrm{d}^*$ is the co-differential operator 
and $\lrcorner$ is the contraction. 
If $\eta$ is co-closed, i.e. $\mathrm{d}^*\eta=0$, then \eqref{ckyManifolds} is equivalent to \eqref{yanoequationoriginal} and $\eta$ is a Killing-Yano $p$-form. If instead $\eta$ is
closed, i.e. $\mathrm{d}\eta=0$, then Semmelmann called such $\eta$ a $*$-Killing $p$-form.
Note that the Hodge $*$-operator 
applied to a conformal Killing $p$-form determines a  conformal Killing $(n-p)$-form on $M$. In particular, the Hodge $*$-operator interchanges closed and co-closed conformal Killing forms on $M$. 

Semmelmann proved that a conformal $p$-form ($p  \neq  3, 4$) on a compact manifold with holonomy $G_ 2$  is parallel. He also showed that the vector space of CKY $p$-forms on a $n$-dimensional connected Riemannian manifold has dimension at most ${n+2}\choose{p+1}$; and there are no conformal Killing forms on compact
manifolds of negative constant sectional curvature.

For a compact simply connected symmetric space $M$,  it was shown in  \cite{BMS} that $M$ admits a non-parallel Killing $p$-form, $p \geq 2$, if and only if it is isometric to a Riemannian product $ S^ k \times N$, where $S^ k$ is a round sphere and $k > p$. Another important result states that every Killing-Yano $p$-form on a compact quaternion K\"ahler manifold is parallel for any $p \geq 2 $, see \cite{Moroianu-semmelmann}. In \cite{Moroianu-semmelmann-08} a description of conformal Killing $p$-forms on a compact Riemannian product was given, proving that such a form is a sum of forms of the following types: parallel forms, pull-back of Killing-Yano forms on the factors, and their Hodge duals.

In order to look for new examples of Riemannian manifolds $(M,g)$ carrying conformal Killing $p$-forms, a good tool is to consider the case when the Riemannian manifold is a Lie group endowed with a left-invariant metric. If we focus on left-invariant forms, then this invariant requirement makes \eqref{yanoequationoriginal} (or \eqref{ckyManifolds}) look more tractable.

Most of the work done about invariant CKY forms is focused on $2$-forms on $2$-step nilpotent Lie groups, flat Lie groups, and Lie groups with bi-invariant metrics.
We briefly explain the main works by chronological order.   

In 2012, Barberis, Dotti and Santill\'an began to study left-invariant Killing 2-forms on Lie groups with left-invariant metric in \cite{BDS}. 
In 2015, Andrada, Barberis and Dotti (\cite{ABD}) worked on conformal Killing 2-forms and obtained some results on Lie groups with bi-invariant metric, $2$-step nilpotent Lie groups and also gave a clasification of 3-dimensional metric Lie algebras admitting a conformal Killing 2-form. In the years 2018-2019,
Andrada and Dotti continued with invariant CKY 2-forms on Lie groups in  \cite{AD}. They gave a  new construction of metric Lie algebras carrying strict conformal Killing 2-forms (CKY not KY). They fully analyzed the case of almost abelian Lie algebras admitting CKY $2$-forms. 
Then in \cite{AD20}, they studied  Killing 2-forms on $2$-step nilpotent Lie groups, proving that non-degenerate Killing $2$-forms appear only on complex $2$-step nilpotent Lie groups with a left-invariant metric.

In 2019 del Barco and Moroianu described Killing 2- and 3-forms on $2$-step nilpotent Lie groups in \cite{dBM19}. They used the de Rham decomposition and reduced  the problem to analyzing only the irreducible components. Their work generalizes some results in \cite{AD20}. In particular, they showed that Killing 2-forms are in correspondence with bi-invariant orthogonal complex structures. In that work they also studied Killing 3-forms, and they proved that left-invariant Killing $3$-forms on irreducible $2$-step nilpotent metric Lie algebras only exist when the corresponding Lie group is naturally reductive. They also looked into the vector space of invariant Killing 2-forms and 3-forms, and classified $2$-step nilpotent Lie groups carrying a left-invariant Riemannian
metric with non-zero Killing 2-forms (up to dimension 8) or Killing 3-forms (up to
dimension 6). 

In 2020, the same authors classified $2$-step nilpotent Lie groups endowed with left-invariant Killing forms of arbitrary degree when the
center of the group is at most $2$-dimensional, see \cite{dBM20}.
Then, in \cite{dBM21}, they proved that for $2$-step nilpotent Riemannian Lie groups with dimension of the center greater than or equal to 4, every conformal Killing 2- or 3-form is Killing. 
They also showed that the only $2$-step nilpotent Lie groups with center of dimension at most 3 admitting left-invariant strict CKY 2- and 3-forms are 
the Heisenberg Lie groups and their trivial
1-dimensional extensions endowed with any left-invariant metric and the simply connected Lie group corresponding to the free $2$-step nilpotent Lie algebra on 3 generators, with a particular 1-parameter family of
metrics.

Finally, in \cite{HO} we presented genuine examples of 5-dimensional Lie groups carrying invariant conformal Killing 2-forms which are not a linear combination of a KY $2$-form and the Hodge dual of KY forms. We also showed a classification of 5-dimensional metric Lie algebras that admits conformal Killing 2-forms when the dual of the $1$-form $\theta$ in \eqref{Tachibanaequationoriginal} is in the center or when it is orthogonal to the center.

Before closing the introduction, let us to introduce the notation for CKY forms in this invariant setting.

\subsection{Left-invariant CKY forms on Lie groups}

Let $G$ be a Lie group and $\ggo$ its Lie algebra of left-invariant vector fields. It is known that there is a linear isomorphism between $\ggo$ and $T_eG$ where $e$ is the identity in $G$. 
Consider a left-invariant Riemannian metric $g$ on $G$, that is $L_a^*g=g$, where $L_a$ denotes the left translation by $a\in G$. There is a well-known correspondence between left-invariant metrics on $G$ and inner products on $\ggo \cong T_eG$ defined as $\la \cdot,\cdot\ra:=g_e(\cdot,\cdot)$. 

Every left-invariant metric $g$  defines a unique Levi-Civita connection $\nabla$ on $G$, which is left-invariant and  for every $x,y,z\in \ggo$ it has the following simple expression:
\begin{equation}\label{koszul}
	2\la \nabla_xy,z\ra=\la [x,y],z\ra-\la [y,z],x\ra+\la [z,x],y\ra.
\end{equation}
In particular, it is easy to see that for any $x\in\ggo$, the endomorphism $\nabla_x:\ggo\to \ggo$ is skew-symmetric with respect to $\la \cdot, \cdot \ra$.

A  $p$-form $\omega$ on $G$ is left-invariant if $L_a^* \omega=\omega$ for all $a\in G$, and every left-invariant $p$-form $\omega$ can be identified with an element in $\Lambda^p\g^*$.
Conversely, every  element in $\Lambda^p\g^*$ defines a left-invariant $p$-form on $G$. 
Since the differential operator $\mathrm{d}$ and the co-differential operator
$\mathrm{d}^*$ preserve left-invariance, they define linear operators on $\Lambda^*\g^*$, which we denote with the same symbols for
simplicity. In particular, the linear operator $\mathrm{d} : \Lambda^p\g^*\to\Lambda^{p+1}\g^*$ is the Lie algebra differential, and $\mathrm{d}^*$ is the 
metric adjoint of $\mathrm{d}$ as soon as $\g$ is unimodular. Therefore, an element $\omega\in\Lambda^p\g^*$ corresponds to a left-invariant conformal
Killing form on $(G, g)$ if and only if
\begin{equation}\label{cky_lie_algebras}
	\nabla_x  \omega=\frac{1}{p+1}x\iprod\mathrm{d}\omega-\frac{1}{n-p+1}x^*\wedge \mathrm{d}^*\omega,
\end{equation}
for all $x\in\g$, where $x\iprod$ denotes the contraction with the vector $x$, and $x^*$ denotes the $1$-form dual to $x$.
We have that a Killing $p$-form is a co-closed $p$-form satisfying \eqref{cky_lie_algebras}. Similarly, a $*$-Killing $p$-form is a closed $p$-form satisfying \eqref{cky_lie_algebras}. Moreover, a strict CKY $p$-form is a CKY $p$-form which is not Killing. The space of solutions of the CKY equation \eqref{cky_lie_algebras} is denoted by $\mathcal{CK}^p(\g,\pint)$, similarly we have  $\mathcal{K}^p(\g,\pint)$ and  $\mathcal{*K}^p(\g,\pint)$ the space of Killing $p$-forms and the space of $*$-Killing $p$-forms respectively. 

\medskip

\section{CKY 2-forms in metric Lie algebras}
\label{xisection}
In this section we will recall the most relevant results about CKY $2$-forms contained in \cite{BDS,ABD,AD}.
Consider a left-invariant CKY $2$-form $\omega$ on a Lie group $G$ endowed with a left-invariant metric $g$, it induces an endomorphism $T$ on $G$ defined by $\omega(X,Y)=g(TX,Y)$ for $X,Y$ vector fields, then $T$ is also left-invariant. In particular, this induces a skew-symmetric endomorphism on $\g$, which we still denote by $T$.
Thus, the   study of left-invariant CKY 2-form is reduced at the Lie algebra level $\left(\mathfrak{g},\langle \cdot,\cdot \rangle\right)$, where $\langle \cdot,\cdot \rangle$ is the inner product induced by $g$. 

Moreover, CKY equation \eqref{cky_lie_algebras} is equivalent to the following expression at the Lie algebra level. 
\begin{equation}
	\label{ckyequationliealgebra}
	\la \left(\nabla_x T\right)y,z\ra+\la \left(\nabla_yT\right)x,z\ra=2\la x,y\ra \theta(z)-\la y,z\ra \theta(x)-\la x,z\ra \theta(y),
\end{equation} 
for some $\theta\in\ggo^*$.   A skew-symmetric endomorphism $T$ that satisfies \eqref{ckyequationliealgebra}  will be called a conformal Killing-Yano (CKY) tensor on $\g$.  
We denote by $\xi$ the unique element of $\ggo $ such that 
\begin{equation}\label{xi}
	\theta(x)=\la \xi,x\ra 
\end{equation}
for all $x\in \ggo$, and we will refer to $\xi$ as the vector associated to $\theta$. Moreover, if $\g$ is unimodular, then $\xi\in\g'=[\g,\g]$ as a consequence of \cite[Lemma 2.3]{ABD}.
Note that a strict CKY tensor is equivalent to requiring that $\theta\neq 0$.

In order to find strict examples of 2-forms CKY, the authors in \cite{AD} studied some properties of those metric Lie algebras carrying such 2-forms. 
In \cite[Proposition 4.1]{AD} they proved that the composition $\theta \circ T=0$. 
As a consequence they obtained that $T\xi=0$, see \cite[Corollary 4.2]{AD}, and showed a strong
algebraic restriction to the existence of a strict CKY tensor.

\begin{teo}\cite[Theorem 4.3]{AD} Let $T$ be a conformal Killing Yano tensor on the metric Lie algebra $\left(\mathfrak{g},\langle \cdot,\cdot\rangle\right)$ If $\theta \neq 0$ then $\dim \mathfrak{g}$ is odd and $T|_{\xi^\perp}:\xi^\perp\to \xi^\perp$ is a linear isomorphism. Moreover $\xi^\perp$ is stable by the operator $\ad_{\xi}$.
\end{teo}
It follows from the above theorem that if the dimension of $\mathfrak{g}$ is even, then $T$ is KY, see \cite[Corollary 4.4]{AD}. In particular, in the 4-dimensional case, any  CKY tensor is parallel, according to \cite[Corollary 4.5]{AD}. This last fact has been proved in a general way in \cite{ABM}, where it is showed that every 4-dimensional Riemannian manifold  with a CKY 2-form of constant norm must be parallel.

The next theorem, which is one of the main results in \cite{AD}, allows us to construct examples of $(n+1)$-dimensional metric Lie algebras admitting CKY tensors starting with a KY tensor with some properties in a $n$-dimensional metric Lie algebra.
\begin{teo}\cite[Theorem 4.6 and Theorem 4.8]{AD}
	\label{teorema dotti-andrada}
	Let $S$ be an invertible KY tensor on the metric Lie algebra $(\mathfrak{h},[\cdot,\cdot],\langle \cdot,\cdot \rangle)$ such that the  $2$-form $\mu(x,y)=-2\la S^{-1}x,y\ra$ is closed. Set $\mathfrak{g}:=\mathfrak{h}\oplus_{\mu} \mathbb{R} \xi $ the central extension of $\mathfrak{h}$ by the $2$-form $\mu$, that is the vector space $\mathfrak{h}\oplus \mathbb{R} \xi $ equipped with the Lie bracket $[\cdot,\cdot]_{\mu}$ defined by
	\begin{equation}\label{extensioncon mu}
		[\mathfrak{h},\xi]_{\mu}=0, \ \ [x,y]_{\mu}= [x,y]+\mu(x,y)\xi,   \  x,y \in \mathfrak{h};
	\end{equation}
	the inner product on $\ggo$ is defined by extending the one on $\mathfrak{h}$ by $\langle \mathfrak{h},\xi\rangle =0$, with $|\xi|>0$ arbitrary. Then, the endomorphism $T$ of $\ggo$ given by  $T|_{\mathfrak{h}}=S$ and  $T\xi=0$ is a strict CKY tensor on $\mathfrak{g}$.
	
	Conversely, any metric Lie algebra  $(\mathfrak{g},\langle \cdot,\cdot \rangle)$  admitting a strict CKY tensor $T$ with associated vector $\xi$ in the center is obtained  in this way, where $\mathfrak{h}=\xi^\perp$, $S:=T|_\mathfrak{h}$, and the Lie bracket on $\mathfrak{h}$ is the $\mathfrak{h}$-component of the Lie bracket on $\ggo$.
	Moreover, the center of $\ggo$ is generated by $\xi$. 
\end{teo}

The first natural examples can be constructed using 4-dimensional metric Lie algebras admitting a invertible Killing (thus parallel) tensor $S$. 
In \cite[Section 3]{Herrera} a full classification of all non-abelian $4$-dimensional metric Lie algebras $(\mathfrak h,\la \cdot, \cdot \ra)$ that carry parallel skew-symmetric endomorphisms was made. 
Moreover, this classification is up to isometric isomorphisms and for each fixed metric Lie algebra all parallel skew-symmetric tensors are given up to equivalence, where,
two parallel skew-symmetric tensors $H_1$ and $H_2$ are said to be \textit{equivalent} if there exists an isometric isomorphism of Lie algebras such that
\begin{equation}\label{equiv}
	\varphi: \ggo\to\ggo \; \text{such that} \; \varphi H_1=H_2 \varphi.	
\end{equation}

Notice that an invertible parallel tensor induces a 2-form $\mu$ as in Theorem \ref{teorema dotti-andrada}, which is always closed, see \cite[Corollary 2.2]{AD}. Using  the classification of invertible parallel tensors in dimension $4$ and Theorem \ref{teorema dotti-andrada},   we constructed 5-dimensional metric Lie algebras carrying a CKY tensor in \cite[Theorem 3.2]{HO}. Moreover, a classification of these $5$-dimensional metric Lie algebras was obtained. In Section \ref{ckylowdimension} we exhibit this classification.

\medskip

\subsection{Bi-invariant metrics}

We recall that a bi-invariant metric on $(G,g)$ is a metric which is invariant under left and right translations. It was proved in \cite{Mi} that a left-invariant metric on a connected Lie group is bi-invariant if and only if the linear transformation $\ad_x$ is skew-symmetric for all $x\in \ggo$. It is also known that a compact Lie group always admits a bi-invariant metric.
We summarize next the most relevant results of CKY tensors on Lie groups admitting a bi-invariant metric.

\begin{lema}\cite[Lemma 4.6]{BDS} Let $(G,g)$  be a Lie group with a  bi-invariant metric. If $T$ is a skew-symmetric tensor on $\g$, then $T$ is  a KY tensor if and only if $T|_{[\mathfrak{g},\mathfrak{g}]}=0$. In particular, a compact semisimple Lie group has no non trivial KY tensor.
\end{lema}

In the particular case of $SU(2)$, there is a stronger condition. Indeed, there are no non-trivial KY tensors for any left-invariant metric
on $SU(2)$, not just the bi-invariant one.

\begin{teo}\cite[Theorem 4.7]{BDS}
For any left-invariant metric $g$ on $SU(2)$, all Killing–Yano tensors are trivial.
\end{teo} 

In \cite{ABD} the authors also consider compact Lie groups with a bi-invariant metric and they prove a very strong restriction for the existence of a strict CKY tensor.
\begin{teo}\cite[Theorem 3.1]{ABD}\label{bi inv dim 3}
Let $G$ be a $n$-dimensional compact Lie group equipped with a bi-invariant metric $g$. If there exists a CKY tensor which is not KY then $n=3$ and $\ggo$ is isomorphic to $\mathfrak{su}(2)$.
\end{teo}

\medskip

\subsection{Flat Lie groups}

Given a Lie group $G$, an inner product $\langle\cdot,\cdot\rangle$ on its Lie algebra $\g$ induces a left-invariant flat metric if and only if:  
 \begin{enumerate}
	\item there exists 
	a  decomposition $\mathfrak{g}=\mathfrak{a}\oplus \mathfrak{u}$, where $\mathfrak{u}$ is an abelian ideal of $\g$ and its orthogonal complement $\mathfrak{a}$ is an abelian subalgebra.
	\item  the endomorphisms $\ad_x$ are skew-symmetric for all $x\in \mathfrak{a}$.
\end{enumerate}
In particular, $\g$ is unimodular. Moreover, $\nabla_u=0$ for all $u\in\u$, and $\dim\g'\geq 2$ if $\g$ is not abelian.

In \cite{AD}, the authors considered  a CKY 2-forms on non-abelian Lie groups with flat left-invariant metric, and they proved that every CKY tensor is parallel.

\begin{lema}\cite[Lemma 6.1]{AD}
	Let $G$ be a non-abelian Lie group with a flat left-invariant metric. If T is a left-invariant CKY tensor, then T is parallel.
\end{lema}

\begin{proof}[Sketch of proof]
	The idea of the proof is to consider $\xi$ as in \eqref{xi}, then $\xi\in\g'\subset\u$.
	If $\xi\neq0$, using the CKY condition \eqref{ckyequationliealgebra} and the fact that $\nabla_u=0$, it is easy to see that $\u=\R\xi$. This is not possible since $\dim\g'\geq 2$, and therefore $\xi=0$ and $T$ is a KY tensor. To show that $T$ is in fact parallel, it reduces to check that $(\nabla_a)b=0$ for all $a,b\in\mathfrak a$, which is a direct computation.
\end{proof}

In \cite{BDS}, the authors had characterized invariant KY tensors on a flat Lie group, which are in fact parallel, as a consequence of the last lemma. Indeed, they proved in \cite[Theorem 4.1]{BDS} that on a Lie group endowed with a flat left-invariant metric, a skew-symmetric endomorphism $T$ on its metric Lie algebra $(\ggo,\langle\cdot,\cdot\rangle)$ is KY (thus parallel) if and only if $[Tx,z]=[\ad_x,T]=0$ for all $x,z\in \mathfrak{a}$.
In particular, any KY tensor on a $(\ggo,\langle\cdot,\cdot\rangle)$  given by
$$T=\left(\begin{array}{cc}T_1&0\\0&T_2
\end{array}\right)$$
where $T_1$ is a skew-symmetric endomorphism in $\mathfrak{a}$ and $T_2$ is a skew-symmetric endomorphism in $\mathfrak{u}$ such that conmutes with $\ad_y$ for all $y\in \mathfrak{a}$, will produce an example of a parallel tensor on $(\ggo,\langle\cdot,\cdot\rangle)$.

As an example, in \cite[Proposition 4.5]{BDS} the authors considered the Lie algebra $\mathfrak{e}(2)\times \mathfrak{e}(2)$\footnote{$\mathfrak{e}(2)$ is the Lie algebra of the isometry group of the Euclidean plane (see Table \ref{tabla_dim3}).}, where  with the orthonormal basis $\{e_1,e_2,e_3,e_4,e_5,e_6\}$ such that
\begin{align*}
	[e_5,e_1]&=e_2, & [e_5,e_2]&=-e_1, & [e_6,e_3]&=e_4, & [e_6,e_4]=-e_3
\end{align*}
In this case a parallel tensor can be written as
$$T=\left( \begin{array}{cccccc}0&-a&&&&\\a&0&&&&\\
	&&0&-b&&\\&&b&0&&\\&&&&0&-c\\&&&&c&0\end{array}\right),\, a,b,c\in \R.$$

\

\subsection{Almost abelian Lie algebras}

We recall that a Lie group $G$ is said to be {\em almost abelian} if its Lie algebra $\g$ has a 
codimension one abelian ideal. Such a Lie algebra will be called almost abelian, and it can be 
written as $\g= \R f_1 \ltimes_{\ad_{f_1}} \mathfrak{u}$, where $\mathfrak u$ is an abelian ideal of 
$\g$, and $\R$ is generated by $f_1$. Accordingly, the Lie group $G$ is a semidirect product 
$G=\R\ltimes_\phi \R^d$ for some $d\in\mathbb N$, where the action is given by 
$\phi(t)=e^{t\ad_{f_1}}$. We point out that an almost abelian Lie algebra is nilpotent if and only 
if the operator $\ad_{f_1}|_{\mathfrak u}$ is nilpotent. 

Regarding the isomorphism classes of almost abelian Lie algebras, one can show that (see \cite{Freibert}):
\begin{lema}\label{ad-conjugated}
	Two almost abelian Lie algebras $\g_1=\R f_1\ltimes_{\ad_{f_1}} \u_1$ and 
	$\g_2=\R f_2\ltimes_{\ad_{f_2}}\u_2$ are isomorphic if and only if there exists $c\neq 0$ such that
	$\ad_{f_1}$ and $c\ad_{f_2}$ are conjugate. 
\end{lema}

In \cite{AD} the authors studied CKY $2$-forms on almost abelian Lie groups, and they proved that only in dimension $3$ it is possible to find examples of strict CKY $2$-forms.
Moreover, if it is a KY $2$-form, then it is parallel.
\begin{teo}\label{almost-abelian-2-forms}\cite[Theorem 6.2]{AD} Let $\mathfrak{g}$ be an almost abelian Lie algebra equipped with an inner product and a CKY tensor $T$.
	\begin{enumerate}
		\item If $\theta  \neq 0$ then $\dim \mathfrak{g}=3$, then $\ggo$ is isomorphic to $\mathfrak{h}_3$  or to $\aff(\R)\times \R$, where $\mathfrak{h}_3$ is the $3$-dimensional Heisenberg Lie algebra and $\aff(\R)$ is $2$-dimensional non abelian Lie algebra (see Table \ref{tabla_dim3}).
		\item If $\theta  =0$ then $T$ is parallel.
	\end{enumerate}	
\end{teo}

\medskip

\section{CKY 2-forms on low dimension}\label{ckylowdimension}

In this section we exhibit the classifications of CKY 2-forms on non-abelian metric Lie algebras in dimensions $3$, $4$, and $5$. 

\medskip

\subsection{Dimension 3}
We summarize the classification obtained in \cite[Section 5]{ABD} in the Table \ref{tabla_dim3}. In the last column of Table \ref{tabla_dim3}, the letter P means ``parallel''. Note also that CKY $2$-forms in a $3$-dimensional Lie algebra are $*$-dual of CKY $1$-forms, which are equivalent to Killing vectors.

{\scriptsize
\begin{table}[h!]
		\begin{tabular}{|c|c|c|c|c|}
			\hline
			\hline
			Lie algebra  & Lie bracket & metric & $\omega$ & {\scriptsize CKY, KY or P} \\ \hline \hline
$\mathfrak{e}(2)$ & $[f_3,f_1]=f_2$, $[f_3,f_2]=-f_1$& $g_t=\left(\begin{array}{ccc} t&0&0\\0&t&0\\0&0&t\end{array}\right)$, $t>0$& $f^1\wedge f^2$& P\\ \hline
			
$\mathfrak{su}(2)$& 
$ [f_1,f_2]=f_3 ,\
 [f_2,f_3]=f_1,\ [f_3,f_1]=f_2 $&
$g_t=\left(\begin{array}{ccc} 1&0&0\\0&1&0\\0&0&t\end{array}\right)$, $t>0$		&  $f^1\wedge f^2$ & CKY \\
\hline
$\mathfrak{sl}(2,\R)$& 
$ [f_1,f_2]=-f_3 ,\
 [f_2,f_3]=f_1,\ [f_3,f_1]=f_2 $

		&
 $g_t=\left(\begin{array}{ccc} 1&0&0\\0&1&0\\0&0&t\end{array}\right)$, $t>0$		&  $f^1\wedge f^2$ & CKY \\

 \hline
$\mathfrak{h}_3$& $[f_1,f_2]=f_3$&  $g_q=\left(\begin{array}{ccc} 1&0&0\\0&1&0\\0&0&q^2\end{array}\right)$, $q>0$& $f^1\wedge f^2$&CKY 
		\\ \hline
		
\multirow{3}{*}{$\mathfrak{aff}(\mathbb{R})\times\R	$} &  \multirow{3}{*}{$[f_1,f_2]=f_2 $}&	 \multirow{3}{*}{$g_{1,t}=\left(\begin{array}{ccc} 1&0&0\\0&1+t^2&t\\0&t&1\end{array}\right)$, $t\geq 0$} &  \multirow{3}{*}{$f^1\wedge f^2$} & \multirow{3}{*}{ P if $t=0$}\\ 
&&&&\\
& & & &   CKY  if $t \neq 0$\\ 
\hline			
		\end{tabular}
	\caption{CKY tensors on non abelian 3-dimensional Lie algebras}
	\label{tabla_dim3}
\end{table} }

\begin{rem}
According to \cite{ABD} all these 2-forms are closed and in the case where $\omega$ is strict CKY, their Hodge dual are contact forms. 
\end{rem}

\medskip

\subsection{Dimension 4}
In \cite[Corollary 4.5]{AD}, the authors concluded that any CKY tensor in a 4-dimensional metric Lie algebra is parallel. This is also a consequence of a more general result,  \cite[Lemma 3.7]{ABM} shows that every CKY 2-form of constant norm  on a 4-dimensional Riemannian manifold is always parallel.  
On the other hand, a classification of parallel tensors on 4-dimensional metric Lie algebras has been done in \cite{Herrera}.  We exhibit this classification in the  Table \ref{tabla_dim4}. The basis $\{e_1,f_1,e_2,f_2,\xi\}$ is orthogonal and $|\xi|=1$. 

{\scriptsize
\begin{table}[h!]
\begin{tabular}{|c|l|c|c|}
\hline\hline

Lie algebra & Lie bracket&metric & parallel $2$-form $\omega$  \\
\hline \hline
  \multirow{2}{*}{$ \R\times \mathfrak{e}(2)$} &\multirow{2}{*}{$\begin{matrix}[e_1 , e_2 ] = -f_2 ,
[e_1 , f_2 ] = e_2\end{matrix}$ }& \multirow{2}{*}{$|e_i|^2=|f_i|^2=t$ }
  &  \multirow{2}{*}{ $a_1e^{1}\wedge f^{1}+a_2e^{2}\wedge f^2$ } \\
  &&&\\
  &&$t>0$&$a_1,a_2\geq 0$\\
 \hline 
\multirow{2}{*}{$\R^2\times \aff(\R)$} &\multirow{2}{*}{$[e_2 , f_2 ] = f_2$}& \multirow{2}{*}{$|e_i|^2=|f_i|^2=t$  }
   & \multirow{2}{*}{$a_1e^{1}\wedge f^{1}+a_2e^{2}\wedge f^2$}\\
   &&&\\
   &&$t>0$& $a_1,a_2\geq 0$  \\ 
 \hline
\multirow{4}{*}{$ \mathfrak{r}'_{4,\lambda,0}$ \;\text{with} \; $\lambda>0$ } & \multirow{2}{*}{$[e_1 , f_1 ] = \lambda f_1, \; [e_1 , f_2 ] = e_2$} & \multirow{2}{*}{$|e_i|^2=|f_i|^2=t$} & \multirow{2}{*}{$a_1e^{1}\wedge f^{1}+a_2e^{2}\wedge f^2$}\\
&&&\\
&$[e_1 , e_2 ] = -f_2$&$t>0$&$a_1,a_2\geq 0$ \\
&&&\\
 \hline
 
 \multirow{4}{*}{$\aff(\R)\times \aff(\R)$}&\multirow{4}{*}{$[e_1 , f_1 ] = f_1, \; [e_2 , f _2 ] = f_2$} & \multirow{2}{*}{$|e_i|^2=t, \,|f_i|^2=ts$} & \multirow{2}{*}{$a_1e^{1}\wedge f^{1}+a_2e^{2}\wedge f^2$}\\
&&\multirow{3}{*}{$s,t>0, \, s\leq 1$} & \multirow{2}{*}{$a_1,a_2\geq 0\, \mbox{if } s<1$}\\
&&&\multirow{2}{*}{$a_1\geq a_2\geq 0\, \mbox{if } s=1$}\\
&&&\\
 
\hline
\multirow{4}{*}{$\mathfrak{d}_{4,\frac{1}{2}}$} & \multirow{2}{*}{$[e_1,f_1]=e_2, \;  [f_2,e_1]=\frac{1}{2}e_1$}&\multirow{2}{*}{$|e_i|^2=|f_i|^2=t$}&\multirow{2}{*}{$c(e^1\wedge f^1-e^2\wedge f^2)$}\\
&&&\\
&$[f_2,f_1]=\frac{1}{2}f_1, \; [f_2,e_2]=e_2$ &  $t>0$ &  $c>0$\\
&&&\\
\hline
 
\multirow{4}{*}{$\mathfrak{d}_{4,2}$}  &\multirow{2}{*}{$[e_1,f_1]=e_2, \; [f_2,e_1]=e_1$} & \multirow{2}{*}{$|e_i|^2=|f_i|^2=t$}&\multirow{2}{*}{$c(e^1\wedge f^2+f^1\wedge e^2)$}\\
&&&\\
&$[f_2,f_1]=-\frac{1}{2}f_1, \;[f_2,e_2]=\frac{1}{2}e_2$ & $t>0$ &  $c>0$\\ 
&&&\\
\hline 

\multirow{4}{*}{$\mathfrak{d}'_{4,\frac{\delta}{2}}$\; \text{with} \; $\delta>0$} &\multirow{2}{*}{$[e_1,f_1]=e_2\; [f_2,e_1]=\frac{1}{2}e_1-\frac{1}{\delta}f_1$}&\multirow{2}{*}{$|e_i|^2=|f_i|^2=t$}&\multirow{2}{*}{$c\left(e^1\wedge f^1-e^2\wedge f^2\right)$}\\
&&&\\
&$[f_2,f_1]=\frac{1}{\delta}e_1+\frac{1}{2}f_1, \; [f_2,e_2]=e_2$ & $t>0$&
$c\neq0$\\
&&&\\
\hline
\end{tabular}
\caption{Parallel $2$-forms on non-abelian 4-dimensional Lie algebras. 
The $2$-form $\omega$	is given in the metric dual basis of $1$-forms		$\{e^1,f^1,e^2,f^2\}$}
\label{tabla_dim4}
\end{table}
}

\medskip

\subsection{Dimension 5}
We studied in  \cite[Theorem 3.2]{HO}  5-dimensional metric Lie algebras admitting CKY 2-form in two cases, namely when $\xi\in \mathfrak{z}$ (see \eqref{xi}) and thus, the dimension of the center $\z$ is one, or when $\dim\z>1$.

In the first case we used Theorem \ref{teorema dotti-andrada} and the classification of parallel skew-symmetric tensors made in Table \ref{tabla_dim4} to construct  5-dimensional metric Lie algebras with  $1$-dimensional center admitting a CKY 2-form whose co-differential lies in the
dual of the center, that is, $\xi\in \mathfrak{z}$.

\begin{teo}\cite[Theorem 3.2]{HO}\label{dim5-extension}
	Let $(\g,\pint)$ be a 5-dimensional metric Lie algebra. If $(\g,\pint)$ admits a strict CKY tensor T with associated vector $\xi\in\z$, then $(\g,\pint)$ is isometrically isomorphic to one and only one of the metric Lie algebras in Table \ref{NotacionNueva}.
	Moreover, the CKY tensor T is uniquely determined, up to scaling, by the corresponding metric Lie algebra and it is given in the last column of Table \ref{NotacionNueva} (as a $2$-form). 
\end{teo}

{\scriptsize
	\begin{table}[h!]
		\begin{tabular}{|c| c| c|}
			\hline \hline
			{\scriptsize Lie algebra }  &{\scriptsize Metric  }&  {\scriptsize strict CKY $2$-form $\omega$} \\ \hline \hline 
			
			$\begin{matrix} (\mathfrak{g}_3,\pint_{t,a_1,a_2})\\ 
				&\\
				[E_1,E_2]=-E_3\\
				[E_1,E_3]=E_2\\
				[E_1,E_4]=-E_5\\
				[E_2,E_3]=-E_5\\
				a_1,a_2,t>0\end{matrix}$
			&
			
			$\begin{pmatrix} 
				t & 0 & 0 & 0 & 0 \\
				0&t& 0&0 & 0\\
				0&0&t&0&0 \\
				0&0& 0& \frac{a_1^2t}{a_2^2}&0 \\
				0&0& 0& 0&\frac{4t^2}{a_2^2} \\
			\end{pmatrix}$
			
			& $\frac{a_1^2 t}{a_2}e^{14}+ a_2te^{23}$ \\
			\hline 
			
			$\begin{matrix} (\mathfrak{g}_2,\pint_{t,a_1,a_2})\\ 
				&\\
				[E_1,E_2]=E_2-E_5\\
				[E_3,E_4]=-E_5\\
				a_1,a_2,t>0\end{matrix}$
			&
			
			$\begin{pmatrix} 
				t & 0 & 0 & 0 & 0 \\
				0&t& 0&0 \\
				0&0&a_1^2t&0&0 \\
				0&0& 0& \frac{t}{a_2^2}&0 \\
				0&0& 0& 0&\frac{4t^2}{a_2^2} \\
			\end{pmatrix}$
			& $ta_2e^{12}+ \frac{ta_1^2}{a_2}e^{34}$ \\
			\hline

			$\begin{matrix} (\mathfrak{g}_8^{\frac1\lambda},\pint_{t,a_1,a_2})\\ 
				&\\
				[E_1,E_4]=-E_1-E_5\\
				[E_2,E_3]=-E_5\\
				[E_2,E_4]=\frac1\lambda E_3\\
				[E_3,E_4]=-\frac1\lambda E_2\\
				\lambda,a_1,a_2,t>0\end{matrix}$
			&
			
			$\begin{pmatrix} 
				(\frac{a_1\lambda }{a_2})^2t & 0 & 0 & 0 & 0 \\
				0&t& 0&0 \\
				0&0&t&0&0 \\
				0&0& 0& \frac{t}{\lambda^2}&0 \\
				0&0& 0& 0&\frac{4t^2}{a_2^2} \\
			\end{pmatrix}$
			& $-\frac{a_1^2 t}{a_2}e^{14}- a_2te^{23}$ \\
			\hline

			$\begin{matrix} (\mathfrak{g}_4,\pint_{s,t,a_1,a_2}) \\ 
				&\\
				[E_1,E_2]=E_2-E_5\\
				[E_3,E_4]=E_4-E_5\\
				t,s>0, s \leq 1 \\
				a_1,a_2> 0  \mbox{ if } s<1, \mbox{ or } \\
				a_1\geq a_2> 0  \mbox{ if } s=1\end{matrix}$
			&
			
			$\begin{pmatrix} 
				t & 0 & 0 & 0 & 0 \\
				0&ta_1^2& 0&0&0 \\
				0&0&ts&0&0 \\
				0&0& 0& \frac{ta_2^2}{s}&0 \\
				0&0& 0& 0& 4t^2 \\
			\end{pmatrix}$
			& $ta_1^2e^{12}+ ta_2^2e^{34}$ \\
			\hline

			$\begin{matrix} (\mathfrak{g}_5,\pint_{t,c})\\ 
				&\\
				[E_1,E_2]=E_3-E_5\\
				[E_1,E_4]=-\frac{1}{2}E_1\\
				[E_2,E_4]=-\frac{1}{2}E_2\\
				[E_3,E_4]=-E_3+E_5\\
				t,c>0\end{matrix}$
			&
			
			$\begin{pmatrix} 
				t & 0 & 0 & 0 & 0 \\
				0&t& 0&0&0 \\
				0&0&t&0&0 \\
				0&0& 0& t&0 \\
				0&0& 0& 0&\frac{4t^2}{c^2} \\
			\end{pmatrix}$
			& $tc(e^{12}-e^{34})$ \\
			\hline

			$\begin{matrix} (\mathfrak{g}_6,\pint_{t,c}) \\ 
				&\\
				[E_1,E_2]=E_3\\
				[E_1,E_4]=-2E_1\\
				[E_2,E_3]=-E5\\
				[E_2,E_4]=E_2\\
				[E_3,E_4]=-E_3\\
				t,c>0\end{matrix}$
			&
			
			$\begin{pmatrix} 
				t+\frac{4t^2}{c^2} & 0 & 0 & 0 & \frac{4t^2}{c^2} \\
				0&t& 0&0 &0\\
				0&0&t&0&0 \\
				0&0& 0& 4t&0 \\
				\frac{4t^2}{c^2}&0& 0& 0&\frac{4t^2}{c^2} \\
			\end{pmatrix}$
			& $tc(2e^{14}+e^{23})$ \\
			\hline

			$\begin{matrix} (\mathfrak{g}_7^{\delta},\pint_{t,c})\\ 
				&\\
				[E_1,E_2]=E_3-E_5\\
				[E_1,E_4]=-\frac\delta2 E_1+E_2\\
				[E_2,E_4]=-E_1-\frac\delta2 E_2 \\
				[E_3,E_4]=-\delta E_3+\delta E_5\\
				\delta,t>0,c\neq 0\end{matrix}$
			&
			
			$\begin{pmatrix} 
				t & 0 & 0 & 0 & 0 \\
				0&t& 0&0&0 \\
				0&0&t&0&0 \\
				0&0& 0& \delta^2t &0 \\
				0&0& 0& 0&\frac{4t^2}{c^2} \\
			\end{pmatrix}$
			& $tc(e^{12}-\delta e^{34})$ \\
			\hline

			$\begin{matrix} (\mathfrak{h}_5,\pint_{a_1,a_2})\\ 
				&\\
				[E_1,E_2]=E_5\\
				[E_3,E_4]=E_5 \\
				a_1,a_2>0\end{matrix}$
			&
			
			$\begin{pmatrix} 
				a_1^2 & 0 & 0 & 0 & 0 \\
				0&\frac{1}{a_2^2}& 0&0&0 \\
				0&0&1&0&0 \\
				0&0& 0& 1&0 \\
				0&0& 0& 0&\frac{4}{a_2^2} \\
			\end{pmatrix}$
			& $\frac{a_1^2}{a_2}e^{12}+a_2e^{34} $ \\
			\hline
			
		\end{tabular}
		\caption{ {\scriptsize The metrics are given in the basis $\{E_1,E_2,E_3,E_4,E_5\}$ and the strict CKY $2$-forms are given in the metric dual basis $\{e^1,e^2,e^3,e^4,e^5\}$ where $e^{ij}=e^i\wedge e^j$. }}
		\label{NotacionNueva}
\end{table}}

Note that all these Lie algebras in Table \ref{NotacionNueva} admit a Sasakian structure for some choice of the parameters.  Indeed, these
Sasakian structures are obtained when the parallel skew-symmetric tensors on $\h=\xi^\perp$ are
in addition complex structures (for example the Sasakian structure on $\g_3$ is obtained with $a_1 = a_2 = 1$ in the first row of Table \ref{NotacionNueva}). Therefore, Theorem \ref{dim5-extension} recovers the classification of Sasakian five-dimensional Lie algebras made in \cite[Section 3.1]{AFV}.  

In \cite{AFV}, the authors also
showed that the simply connected Lie group associated to the unimodular Lie algebra
$\g_3$ admits lattices, that is, co-compact discrete subgroups. The other unimodular Lie
algebra in Table \ref{NotacionNueva} is $\h_5$, and it is known that the Heisenberg Lie group
$H_5$ with Lie algebra $\h_5$ also admits lattices.

We also studied the space of solutions of CKY $2$-forms $\mathcal{CK}^2(\g,\pint)$ for those metric Lie algebras.
From Theorem \ref{dim5-extension} we have that there is only one strict CKY 2-form,
up to scaling. Indeed, any non-zero multiple of a strict CKY tensor is again a strict CKY tensor.
In the next result, we showed that each metric Lie algebra in Table \ref{NotacionNueva} does not admit
any other CKY 2-forms, including KY 2-forms and strict CKY 2-forms whose co-differential is not necessary in the dual of the center.

\begin{teo}\cite[Theorem 6.2]{HO}
Let $(\g,\pint)$ be a 5-dimensional metric Lie algebra admitting a
strict CKY 2-form, such that its co-differential lies in the dual of the center, then
$\mathcal{K}^2(\g,\pint)= 0$ and $\mathcal{*K}^2(\g,\pint)=\mathcal{CK}^2(\g,\pint)$ is 1-dimensional. In particular, $(\g,\pint)$ does not admit non-zero parallel tensors.
\end{teo}

For the second case, that is $\dim\z>1$,
we classified in  \cite[Theorem 3.2]{HO} the $5$-dimensional metric Lie algebras with
center of dimension greater than one admitting strict CKY tensors. In addition, we
determined all possible CKY tensors on these metric Lie algebras. In particular, we
exhibited the first examples of CKY 2-forms on metric Lie algebras which do not admit
any Sasakian structure.

It is easy to see that if a metric Lie algebra admits a strict CKY tensor and $\dim\z>1$, then $\xi\perp\z$  \cite[Lemma 4.1]{HO}. Moreover, if $\g$ is 5-dimensional, then $\dim\z\leq3$. We proved in \cite[Lemma 5.1]{HO}
that a 5-dimensional metric Lie algebra $(\g,\pint)$ with $\dim \z = 3$ does not
admit any strict CKY tensor. Therefore, the only interesting case is $\dim\z=2$. In this case we showed that:

\begin{teo}\cite[Theorem 5.4]{HO}\label{dim5-center2}
	Let $(\g,\pint)$ be a $5$-dimensional metric Lie algebra with $\dim\z>1$. If $(\g,\pint)$ admits a strict CKY tensor, then $(\g,\pint)$ is isometrically isomorphic to one and only one of the metric Lie algebras in Table \ref{tabla_dim5_xi_no_central}.
	Moreover, the CKY tensor is uniquely determined by the metric, up to scaling.
\end{teo}
\begin{tiny}
\begin{table}[h!]
	\begin{tabular}{|c | c | c| }
		\hline \hline
		{\scriptsize   Lie algebra }   &{\scriptsize  Metric } & {\scriptsize strict CKY $2$-form} \\ \hline \hline 
		
			$\begin{matrix} (L_{5,9},\pint_r):\\ 
			&\\ 
				[X,E]=Z_1, &
				[Y,E]=Z_2, \\
				[X,Y]=E,
				&
				r>0\end{matrix}$
	
& 
$\begin{pmatrix}
	r^2&0&0&0&0\\
	0&r^4&0&0&0\\
	0&0&r^4&0&0\\
	0&0&0&1&0\\
	0&0&0&0&1
\end{pmatrix}$
&

			 $\begin{matrix}6r^3z^1\wedge z^2 \\- 2r(z^1\wedge x^* + z^2\wedge y^*) \\
			+ \frac{4}{r}x^*\wedge y^*\end{matrix}$
		\\ \hline
						
			$\begin{matrix}
				\left(\R^2 \times\mathfrak{su}(2),\pint_{r,s}\right): \\ & \\
				[E,X]=Y, \\
				[X,Y]=E\\ 
				[Y,E]=X, \\ a_4=\frac{s^2-r^2}{s},\,s>r>0\end{matrix}$
	
		&
		$\begin{array}{c}\begin{pmatrix}
	\frac{1}{a_4^2}&0&0&0&0\\
	0&1&0&\frac{r}{\sqrt{a_4^3s}}&0\\
	0&0&1&0&-\frac{r}{\sqrt{a_4^3s}}\\
	0&\frac{r}{\sqrt{a_4^3s}}&0&\frac{r^2+a_4^2}{a_4^3s}&0\\
	0&0&-\frac{r}{\sqrt{a_4^3s}}&0&\frac{r^2+a_4^2}{a_4^3s}
\end{pmatrix}\end{array}$
		&

			 $\begin{matrix}\frac{2(s^2+2r^2)}{r^2s} z^1\wedge z^2 - \\ \frac{6r}{(s^2-r^2)^{3/2}}(z^1\wedge y^* + z^2\wedge x^*)  \\ -\frac{2(s^2+2r^2)}{s(s^2-r^2)^3}x^*\wedge y^*\end{matrix}$		
		
	\\ \hline

			$\begin{matrix}
				\left(\R^2 \times\mathfrak{sl}(2,\R),\pint_{r,s}\right):\\&\\
				[E,X]=Y\\
				[X,Y]=-E\\ 
				[Y,E]=X\\
a_4=\frac{s^2-r^2}{s}, \, 
		r>s>0\end{matrix}$

		& $\begin{array}{c}
		\begin{pmatrix}
	\frac{1}{a_4^2}&0&0&0&0\\
	0&1&0&\frac{r}{\sqrt{|a_4|^3s}}&0\\
	0&0&1&0&\frac{r}{\sqrt{|a_4|^3s}}\\
	0&\frac{r}{\sqrt{|a_4|^3s}}&0&\frac{r^2+a_4^2}{|a_4|^3 s}&0\\
	0&0&\frac{r}{\sqrt{|a_4|^3s}}&0&\frac{r^2+a_4^2}{|a_4|^3s}
\end{pmatrix}\end{array}$
		&
			 $\begin{matrix}\frac{2(s^2+2r^2)}{r^2s} z^1\wedge z^2 +  \\
			 \frac{6r}{(r^2-s^2)^{3/2}}(z^1\wedge y^* - z^2\wedge x^*)  \\ -\frac{2(s^2+2r^2)}{s(s^2-r^2)^3}x^*\wedge y^*\end{matrix}$
						
		 \\
	 \hline
	\end{tabular}
	\caption{Notation: $\{e,z^1,z^2,x^*,y^*\}$ is the metric dual basis of $\{E,Z_1,Z_2,X,Y\}$.}
	\label{tabla_dim5_xi_no_central}
\end{table}
\end{tiny}

\begin{proof}[Sketch of proof]
	Let  $(\g,\pint)$ a $5$-dimensional metric Lie algebra with $\dim\z=2$ admitting an strict CKY tensor $T$.
	We can decompose $\g=\z\oplus\z^{\perp}$ with $\xi \in \z^{\perp}$, and $\xi^{\perp}=\z\oplus(\z^\perp\cap\xi^\perp)$.
	Consider the operators  $T | _{\xi^ {\perp}}$ and $\ad_\xi| _{\xi^ {\perp}}$ according to this decomposition. 
	Then, taking an orthonormal basis of $\xi^ {\perp}$ we have that the Lie brackets on $\g$ are encoded in a $3\times 5$ matrix (since $\dim\z=2$) and $T$ can be written as a skew-symmetric matrix determined by $6$ variables (since $T\xi=0$).
	
After solving the CKY equation we reduce the structure constants to only two parameters $r,s$, and $T$ is
uniquely determined by $r,s$. Equivalently, $(\g,\pint)$ is isometrically isomorphic to one and only one $(\g_{r,s} , \pint)$ with Lie brackets 
\begin{align}
	[\xi,x]&=rz_1+a_4y,&
	[\xi,y]&=rz_2-a_4x,&
	[x,y]&=s\xi,
\end{align}
where $r, s>0$, and  $a_4=\frac{s^2-r^2}{s}$. The $2$-form associated to $T$ is given by
\begin{equation*}
\omega=\frac{2(s^2+2r^2)}{r^2s}z^1\wedge z^2 + \frac{2}{r}(z^1\wedge x^* +z^2\wedge y^*) + \frac{4}{s}x^*\wedge y^*.
\end{equation*}

It can be shown that $\g_{r,s}$ are unimodular for all $r,s$ and 
$(\g_{r,s} , \pint)$ are pairwise non-isometrically isomorphic for $r,s>0$.
Finally, we consider three cases in order to get the final reduction. Indeed,  	if $r=s$ we obtain that $(\g_{r,s} , \pint)$ is isometrically isomorphic to $(L_{5,9},\pint_r)$. Simillarly, if $s>r>0$  we obtain $(\R^2 \times\mathfrak{su}(2),\pint_{r,s})$ and if  $r>s>0$ we obtain $(\R^2 \times\mathfrak{sl}(2,\R),\pint_{r,s})$.
\end{proof}

It is important to point out that these 5-dimensional metric Lie algebras exhibited in Theorem \ref{dim5-center2} represent the first explicit examples of Lie algebras carrying strict CKY 2-forms and not admitting any Sasakian structure. Note also that $L_{5,9}$ is a $3$-step nilpotent Lie algebra.

Moreover, the simply connected Lie groups associated to the  Lie algebras in Table \ref{tabla_dim5_xi_no_central} admit lattices, and therefore they can be used to
produce examples of compact manifolds admitting invariant CKY tensors. Indeed, lattices in the simply connected Lie group associated to $L_{5,9}$ were studied in \cite{Malcev}. Therefore, any nilmanifold obtained as a quotient of the simply Lie group connected Lie group with Lie algebra $L_{5,9}$
admits a CKY tensor induced by $T$, but  does not admit any Sasakian structure, since it is shown in \cite{CMdNMY}  that the only nilmanifolds admitting a Sasakian structure (not necessarily invariant) are quotients of the
Heisenberg group. 

For the compact manifold $\mathcal T_2 \times \mathcal S^3$ obtained from $\R^2 \times\mathfrak{su}(2,\R)$, where $\mathcal S^3$ is the 3-dimensional sphere and $T_2$ is the 2-dimensional torus, we know that this Lie
group does admit Sasakian structures (which are non-invariant), see \cite{BTF}. Finally, the Lie group $SL(2, \R) \times \R^2$ admits
lattices (see \cite{Mos75}), but it is not known in this case if the induced compact manifolds
admit non-invariant Sasakian structures.

On the other hand, for those  5-dimensional metric Lie algebras exhibited in Theorem \ref{dim5-center2}
we also analyzed the space of solutions of CKY $2$-forms $\mathcal{CK}^2(\g,\pint)$ and we obtained:

\begin{teo}\cite[Theorem 6.5]{HO}
	For any metric Lie algebras $(\g,\pint)$ in Theorem \ref{dim5-center2}, we have that $\mathcal{K}^2(\g,\pint)=0$ and $\mathcal{*K}^2(\g,\pint)=0$ whereas
	$\mathcal{CK}^2(\g,\pint)$ is $1$-dimensional.
\end{teo}

\begin{proof}[Sketch of the proof]
	Let $(\g,\pint)$ be any metric Lie algebra from Theorem \ref{dim5-center2}.
	We proved in that theorem  that there is only one, up to scaling, strict CKY $2$-form for $(\g,\pint)$. Moreover, it is not hard to see that this strict CKY $2$-form  is never closed, thus it is not $*$-Killing. To complete the proof we need to see that $(\g,\pint)$ does not admit any KY tensor.

	Assume $T$ is a KY tensor on $(\g,\pint)$ thus $T$ preserves $\z$ and $\z^\perp$ and therefore $T$ induces a KY tensor on $\mathfrak h=\g_{r,s}/\z$ which is isomorphic to $\h_3,\mathfrak{su}(2)$ or $\mathfrak{sl}(2,\R)$. 
	This contradicts the fact that any $3$-dimensional metric Lie algebra admitting a KY $2$-form is isomorphic to $\mathbb{R}\times \aff(\R)$ or $\mathfrak{e}(2)$ according to Table \ref{tabla_dim3}. 
\end{proof}

\section{Nilpotent Lie groups}

If we restrict ourselves to nilpotent Lie groups, we have the following important result about the geometry of nilpotent Lie groups.
It states that the de Rham decomposition of $(N,g)$ corresponds to the decomposition of $(\n,\pint)$ into irreducible orthogonal ideals.
It was shown first in \cite[Corollary A.4]{dBM20sym} for $2$-step nilpotent Lie groups, and it was generalized later in \cite{dBM20} to nilpotent Lie groups of arbitrary nilpotency degree.
\begin{teo}\cite[Theorem 2.4]{dBM20}\label{deRham}
	Let $(N, g)$ be a connected and simply connected nilpotent Lie group endowed with a left-invariant Riemannian metric, and consider its de Rham decomposition
	\begin{equation*}
		(N,g)=(N_0,g_0)\times(N_1,g_1)\times\cdots\times(N_r,g_r),
	\end{equation*}
	where $(N_0, g_0)$ is the Euclidean space and $(N_i, g_i)$ are irreducible Riemannian manifolds. Then each $(N_i, g_i)$, with $i =1, \dots, r$, is (isometric to) a connected, simply connected irreducible nilpotent Lie group endowed with a left-invariant metric. In particular, the Lie algebra $\n$ of $N$ is a direct sum of orthogonal ideals
	\begin{equation}
		(\n,\pint)=(\mathfrak a,\pint_0)\oplus \bigoplus_{i=1}^m(\n_i,\pint_i),
	\end{equation}
	where $\mathfrak a=T_eN_0$ is abelian and $\n_i=T_eN_i$ is nilpotent, non-abelian and irreducible for $i =1, \dots, r$.
\end{teo}

We also recall two results regarding left-invariant KY $p$-forms on Lie groups.
The first proposition is a Lie group analogue of a result proved in \cite{Moroianu-semmelmann-08} about the decomposition of Killing forms on a product of compact Riemannian manifolds. Let $(N_1, g_1)$, $(N_2, g_2)$ be Lie groups endowed with a left-invariant metric and consider $N=N_1\times N_2$ endowed with the product metric. Then $\n =\n_1\oplus \n_2$ is an orthogonal direct sum of ideals, where $\n_i=T_eN$. Let $\alpha=\sum_l \alpha_l$ with $\alpha_l\in\Lambda^l\n_1^*\otimes\Lambda^{k-l}\n_2^*$ for $l=0,\dots, k$, then we have:

\begin{prop}\cite[Proposition 3.2]{dBM20}
	A left-invariant k-form $\alpha$ is a Killing form on $(N, g)$ if and only if $\alpha_0$ and $\alpha_k$ are Killing forms on $N_2$ and $N_1$ respectively, and $\alpha_l$ is a parallel form on $N$ for each $l=1,\dots, k-1$.
\end{prop}

The last result was proved first  in \cite{dBM19} for left-invariant Killing 2-and 3-forms on $2$-step nilpotent Lie groups and extended later in \cite{dBM20} to forms of arbitrary degree and to arbitrary Lie groups.

The second result is about parallel forms and shows that:
\begin{prop}\cite[Proposition 3.3]{dBM20}
	Let $(N, g)$ be a de Rham irreducible connected and simply connected (non-abelian) nilpotent Lie group endowed with a left-invariant metric. Then the only parallel differential forms on $(N, g)$ are the constants and the constant multiples of the volume form. In particular, (N, g) does not admit K\"ahler structures.
\end{prop}

As a consequence, it is shown the following decomposition result for left-invariant Killing forms on nilpotent Lie groups.
\begin{cor}\cite[Corollary 3.5]{dBM20}
	Every left-invariant Killing form on a connected and simply connected nilpotent Lie group is the sum of Killing forms on its de Rham factors, and a left-invariant parallel form. The latter is a linear combination of wedge products of volume forms of some of the irreducible de Rham factors and of any left-invariant form on the flat factor.
\end{cor}

\medskip

\subsection{$2$-step nilpotent metric Lie algebras}
A $2$-step nilpotent Lie algebra $\n$ is a non-abelian Lie algebra such that $\ad_x^2=0$ for all $x\in\n$. Equivalently, $\n$ is $2$-step nilpotent if its commutator $\n'=[\n,\n]$ is non-trivial and is contained in the center $\z$ of $\n$ ($[\n,[\n,\n]]=0$).
Let consider an inner product $\pint$ on $\n$, then $\n$ can be written as $\n=\mathfrak{z}\oplus \mathfrak{v}$, where $\mathfrak{v}$ is the orthogonal complement of the center.

Let's consider an inner product $\pint$ in $\n$, then each $z\in\z$ defines a skew-symmetric endomorphism $j(z) : \v \to \v$ given by
\begin{equation}\label{jz}
	\langle j(z)x,y\rangle=\langle z, [x,y]\rangle,
\end{equation}
for all $x,y\in\v$. Since $\n'\subset\z$, the Lie algebra structure of $\n$ is completely determined by the map $j : \z \to \mathfrak{so}(\v)$.

Let $N$ be a simply connected Lie group corresponding to a $2$-step nilpotent Lie algebra $\n$, and $g$ the left-invariant metric of $N$ induced by $\pint$. The Levi-Civita connection on $(N,g)$ defines a linear map
$\nabla : \n \to \mathfrak{so}(\n)$ which by Koszul's formula \eqref{koszul} satisfies
$$\nabla_xy=\frac12([x,y]-\ad_x^*y -\ad_y^* x),$$
for all $x,y\in\n$, where $\ad_x^*$ denotes the adjoint of $\ad_x$ with respect to $g$. Now, using the endomorphism $j$ in \eqref{jz}, we have that the covariant derivative satisfies 
\begin{equation}\label{LC-j}
\Biggl\{
\begin{array}{l l}
	\nabla_xy=\frac12[x,y], & \quad\text{for} \quad x, y \in\v \\
	\nabla_xz=\nabla_zx=-\frac12 j(z)x, &\quad\text{for} \quad x \in\v, \quad z\in\z\\
	\nabla_zz'=0, &\quad\text{for} \quad z,z' \in\z .
\end{array} 
\end{equation}
In \cite{dBM19} the authors proved that $\v = \sum_{z\in\z} \Im j(z)$.
The linear map $j : \z \to \mathfrak{so}(\v)$  is injective if and only if the commutator of $\n$
coincides with its center, $\n'=\z$. On the other hand, if $\mathfrak a=\ker j$ and $\mathfrak a^\perp$ is its orthogonal complement in $\z$, then $\mathfrak a $ is an abelian ideal of $\n$ and $\n_0:=\v\oplus\mathfrak a ^\perp$ is a $2$-step nilpotent ideal of $\n$, such that $\n_0'=\n'=\mathfrak a^\perp$. 
Since the Levi-Civita connection is zero on elements from $\mathfrak a$, see \eqref{LC-j}, then $\mathfrak a $ corresponds to the flat Riemannian factor in the de Rham decomposition on $N$.

\

\subsubsection{Space of $p$-forms}
Note, that the orthogonal decomposition  of a $2$-step Lie algebra $\n=\mathfrak{z}\oplus \mathfrak{v}$ also induces a decomposition of the space of $p$-forms on $\n$
\begin{equation}
	\Lambda^p\n^*=\bigoplus_{k=0}^p \Lambda^k\v^*\otimes\Lambda^{p-k}\z^*.
\end{equation}
Moreover, we have a further direct sum decomposition $\Lambda^p\n^*=	\Lambda^p_{ev}\n^* \oplus	\Lambda^p_{odd}\n^*$ where
\begin{equation}\label{v-degree}
	\Lambda^p_{ev}\n^*=\bigoplus_{k\; \text{even}} \Lambda^k\v^*\otimes\Lambda^{p-k}\z^* \quad \text{and} \quad
	\Lambda^p_{odd}\n^*=\bigoplus_{k\; \text{odd}} \Lambda^k\v^*\otimes\Lambda^{p-k}\z^*.
\end{equation}
If a $p$-form $\alpha\in\Lambda^p_{ev}\n^*$, then $\alpha$ is of even $\v$-degree, and if $\alpha\in\Lambda^p_{odd}\n^*$, then $\alpha$ is of odd $\v$-degree. 
Correspondingly, we have the space of CKY (KY) $p$-forms of even or odd $\v$-degree, that is, 
\begin{equation}
	\mathcal K^p_{ev}(\n,\pint)\subset\mathcal{CK}^p_{ev}(\n,\pint) \subset\Lambda^p_{ev}\n^* \quad \text{and} \quad
	\mathcal K^p_{odd}(\n,\pint)\subset\mathcal{CK}^p_{odd}(\n,\pint) \subset\Lambda^p_{odd}\n^*.
\end{equation}
It can be shown that the even and odd components of every CKY form are again CKY, therefore we have the space of CKY $p$-forms decomposed into its even and odd parts, that is,
\begin{equation}
	\mathcal{CK}^p(\n,\pint) = \mathcal{CK}^p_{ev}(\n,\pint)\oplus\mathcal{CK}^p_{odd}(\n,\pint).
\end{equation}
Similarly, we have a decomposition of the space of KY $p$-forms
\begin{equation}
	\mathcal{K}^p(\n,\pint) = \mathcal{K}^p_{ev}(\n,\pint)\oplus\mathcal{K}^p_{odd}(\n,\pint).
\end{equation}

\medskip

\section{Killing $2$-forms on $2$-step nilpotent Lie algebras}

The work of KY $2$-forms on $2$-step nilpotent Lie algebras began in  \cite{BDS}, a few years later it continued with the works of Andrada-Dotti in \cite{AD20} and del Barco-Moroianu in \cite{dBM20}.
We start by recalling the main result in \cite{BDS}:
\begin{teo}\label{teoKY2nilpotente}
	\cite[Theorem 3.1]{BDS} Let $T$ be a skew-symmetric endomorphism on a $2$-step nilpotent Lie algebra $\n$ endowed with an inner product $\langle \cdot,\cdot\rangle$. Then $T$ is a KY tensor if and only if $T$ preserves the center and for all $x,y \in \mathfrak{v}$, it holds:
	\begin{align}
	\label{Ky2pasos}
	[Tx,y]=[x,Ty] \\
	T[x,y]=3[Tx,y]
	\end{align}
	or equivalently
	\begin{align*}
		\ad_x\circ T=\ad _{Tx}=\frac{1}{2}[T,\ad_{x}]
	\end{align*}
\end{teo}
Notice that if  $T$ preserves the center, it will also preserves $\mathfrak{v}$. The conditions in Theorem \ref{teoKY2nilpotente} can be described in terms of the maps $j_z$ introduced in \eqref{jz}, that is, a skew-symmetric endomorphism $T$ on a $2$-step nilpotent metric Lie algebra $(\n,\pint)$ is Killing if and only if $\frac13j_{Tz}=T\circ j_z=-j_z\circ T$, for all $z\in\z.$ Now, using Theorem \ref{teoKY2nilpotente},  the next result follows.
\begin{cor}\cite[Corollary 3.3]{BDS}\label{KY en heisenberg} If $T$ is a KY tensor on the $(2n+1)$-dimensional Heisenberg Lie algebra, then
	$T$ is trivial.
\end{cor}

For complex Heisenberg Lie groups we have the next example in complex dimension $3$ which appears in \cite{BDS}.
\begin{ejemplo}\label{complex-2step}
	Consider the 3-dimensional complex Heisenberg group $N$ given by upper triangular complex $3\times 3$ matrices with $1's$ in the diagonal. Taking the standard left-invariant  
	Riemannian metric $g$, then $\n$ admits an orthonormal basis $\{e_1,e_2,e_3,e_4,e_5,e_6\}$ such that 
	\begin{align*}
		[e_1,e_3]&=-[e_2,e_4]=e_5, & [e_1,e_4]=[e_2,e_3]=e_6,
	\end{align*}
	the standard complex structure is defined by $J e_{2i-1}=e_{2i}$ for $i=1,2,3$ and $J^2=-I$.
	In \cite[Corollary 3.4]{BDS} it is shown that the following non-degenerate tensor $T$ is a Killing-Yano tensor 
	$$T=\left( \begin{array}{cccccc}0&-3a&&&&\\3a&0&&&&\\
		&&0&-a&&\\&&a&0&&\\&&&&0&-a\\&&&&a&0\end{array}\right),\, a\in \R.$$	
	Moreover, it is the only left-invariant KY tensor on $(N,g,J)$, up to scaling.
\end{ejemplo}
\begin{rem}
The KY tensor above is not parallel. Indeed, any left-invariant invertible KY tensor on a $2$-nilpotent Lie
group is not parallel. This is a consequence of \cite[Theorem 5.1]{AD}, which states that there are no skew-symmetric parallel tensor on a Lie algebra $\g$ satisfying $\g'\cap\z\neq\{0\}$.
\end{rem}

In \cite{BDS}, the authors also considered a lattice $\Gamma$ in $N$ (discrete and co-compact subgroup) given by upper triangular complex $3\times 3$ matrices with integer coefficients and $1's$ in the diagonal. The compact manifold $\Gamma\setminus N$ is known as the Iwasawa manifold, endowed with the induced left-invariant metric admits a KY tensor.
Finally, they exhibited a natural extension of the complex Heisenberg group to higher dimensions and proved that the corresponding $2$-step nilpotent Lie algebras admit KY tensors.
	
\medskip

Motivated by this example of KY tensors on the complex $3$-dimesional Heisenberg Lie group, Andrada and Dotti 
studied KY tensors on $2$-step nilpotent Lie algebras in \cite{AD20}. First, as a consequence of Theorem \ref{teoKY2nilpotente} it is proved in \cite[Corollary 3.2]{AD20} that there are no nearly K\"ahler structures on $2$-step nilpotent Lie algebras. Then, they showed that it is enough to study invertible KY tensors on $2$-step nilpotent Lie groups, this is a consequence of the following result.

\begin{teo}\cite[Theorem 3.3]{AD20}
	Let $\left(\mathfrak{n},\langle \cdot,\cdot \rangle\right)$ be a $2$-step nilpotent metric Lie algebra. If $T$ is a KY tensor on $\mathfrak{n}$ then:
	\begin{enumerate}
		\item $\mathfrak{n}$ is isometrically isomorphic to a direct product of ideals $\mathfrak{n}=\mathfrak{n}_1\times \mathfrak{n}_2$, where $T|_{\mathfrak{n}_1}=0$, $\mathfrak{n}_2$ is $T$-invariant and $T|_{\mathfrak{n}_2}$ is a KY tensor on $\mathfrak{n}_2$.
		\item $T$ is parallel if and only if  $\mathfrak{n}_2$ is abelian. Moreover if $\mathfrak{z}=\mathfrak{n}'$ then $T=0$.
	\end{enumerate}
\end{teo}
\begin{rem}In a $2$-step nilpotent metric Lie algebra carrying a KY tensor, Ker $T$ and Im $T$ are ideals of $\mathfrak{n}$. 
\end{rem}

In  \cite[Proposition 3.5]{AD20} the authors showed that any complex $2$-step nilpotent Lie group $N$ endowed with a left-invariant Hermitian metric admits a non-parallel invertible KY tensor. 
Moreover, such a tensor is constructed using the bi-invariant complex structures on $N$. Indeed, if we decompose orthogonally $\mathfrak{z}=\mathfrak{a}\oplus \mathfrak{n}'$, then one can define a KY tensor by 
$T|_\v=J|_\v$, $T|_{\n'}=3J|_{\n'}$, $T|_{\mathfrak a}=T_0,$
where $T_0:\mathfrak a \to \mathfrak a$ is any skew-symmetric isomorphism of $\mathfrak a$.

Conversely, they proved that the only $2$-step nilpotent Lie groups with a left-invariant metric carrying a non-parallel invertible left-invariant KY $2$-form are the complex Lie groups.

\begin{teo}\label{decomposicion12nilpotente}
	\cite[Theorem 3.7]{AD20} If a $2$-step nilpotent Lie group N equipped with a left-invariant metric g admits an invertible left-invariant KY tensor then N is a complex Lie group, that is, there exists a bi-invariant complex structure. Moreover $g$ is Hermitian with respect to this complex structure. 
\end{teo}

They also showed that for $2$-step nilpotent complex Lie groups $\mathbb{G}$ arising from graphs the space of KY tensors is  $1$-dimensional, see \cite[Theorem 4.1]{AD20}. Furthermore, if the complex Lie group is connected then any non-zero KY tensor is invertible, thus non-parallel.
In \cite[Proposition 4.9]{dBM19} this was generalized to any irreducible $2$-step nilpotent metric Lie algebra. Indeed, it is proved that an irreducible $2$-step nilpotent metric Lie algebra admits (up to sign) at most one bi-invariant orthogonal complex structure.

\medskip

Note that \cite[Proposition 3.5]{AD20}  and Theorem \ref{decomposicion12nilpotente} give a correspondence between invertible left-invariant KY tensor and bi-invariant complex structures on $2$-step nilpotent Lie algebras.
This result was also proved by del Barco and Moroianu in a different way. 

\begin{prop}\cite[Proposition 4.7]{dBM19}\label{J bi-inv dB-M}
	On an irreducible $2$-step nilpotent metric Lie algebra, there is a one to
	one correspondence between non-zero KY 2-forms (up to scaling), and orthogonal bi-invariant complex structures.
\end{prop}

We will explain the main ideas of both approaches to prove  Theorem \ref{decomposicion12nilpotente} and Proposition \ref{J bi-inv dB-M}.

\begin{proof}[Sketch of proof] 
In Theorem \ref{decomposicion12nilpotente} the strategy is to use the spectral decomposition associated to a given Killing tensor $T$. Indeed, taking into the account the decomposition $\n=\mathfrak{z}\oplus \mathfrak{v}=\mathfrak a\oplus\n'\oplus\v$, and the fact that $T$ preserves them, there exist an orthonormal basis $\{e_i,f_i\}$ for $i=1,\dots, n$ with $2n=\dim\v$
such that 
$Te_i=a_if_i$, $Tf_i=-a_if_i$, with $ 0 <a_1 \leq a_2 \leq\dots\leq a_n$. 

Then, they defined $\mathfrak{v}_i = \operatorname{span}\{e_i , f_i \}$, thus $\mathfrak{v}=\oplus_{i=1}^n \mathfrak{v}_i$. Using, \eqref{teoKY2nilpotente} it can be shown that for all $i$, $[\mathfrak{v}_i,\mathfrak{v}_i]=0$;  for each $i$ there exists $j$ such that $[\mathfrak{v}_i,\mathfrak{v}_j]\neq 0$, thus $a_i=a_j$; $[\mathfrak{v}_i , \mathfrak{v}_j ]$ is $2$-dimensional and $T$-invariant. 

It suggests to define the set of integers
$\{i_j : j = 1, \dots , r\} $ such that
$i_1 = 1, i_j = \text{min}\{k:a_k  \neq a _{i_{j-1} }\}$, and  $W_j=\bigoplus_{i:a_i=a_{i_j}}\mathfrak{v}_i\subset\mathfrak{v}$. Moreover, $\mathfrak{v} = \bigoplus_{ j} W _j$ and $[W_j , W_k ] = 0$ if $j\neq k$. Then, $\mathfrak{n}_j = W_j \oplus [W_j , W_j ]$ are ideal ideals of $\n$,   $\mathfrak{n}'=\bigoplus_j [W_j , W_j ]$, and $\mathfrak{n}'\oplus \mathfrak{v}=\oplus_{j}\mathfrak{n}_j$.
Therefore, $\mathfrak{n}$ is decomposed into a direct sum of orthogonal $T$-invariant ideals
	$$\mathfrak{n}=\mathfrak{a}\oplus\mathfrak{n}_1\oplus \cdots \oplus \mathfrak{n}_r.$$
Now, it can be define \color{black}a bi-invariant complex structure on $\n$ such that:
$ J|_{\mathfrak{a}}$ is any almost complex structure compatible with $\pint$, 
$J|_{W_ j} = \frac{1}{a_{ i_j} }T|_{W_j}$ and $ J|_{[W _j ,W_j ]} =\frac{1}{3a_{i_j}}T|_{[W_j ,W_j ]}$.

On the other hand, the main idea in Proposition \ref{J bi-inv dB-M} is to use the de Rham decomposition in Theorem \ref{deRham} to reduce the problem to the irreducible case. 
Using \cite[Proposici\'on 4.1]{dBM19} a Killing 2-form $\alpha$ on $(\mathfrak{n},\pint)$ can be decomposed as  $\alpha=\alpha_0+\alpha_2$,  where $\alpha_0$, $\alpha_2$ are 2-forms on $\mathfrak{z}$ and $\mathfrak{v}$, respectively.

Let $S$ be a symmetric endomorphism $S$ associated to the Killing 2-form $\alpha$ defined such that $S$ preserves $\mathfrak{z}$ and its ortogonal complement $\mathfrak{v}$, and 
$S|_{\mathfrak{z}}=\frac{1}{9}{\alpha_0}^2$ and $S|_{\mathfrak{v}}={\alpha_2}^2$.
The irreducibility condition on $\mathfrak{n}$, 
implies $S=\lambda Id_{\mathfrak{n}}$. Moreover, $\lambda<0$ if $\alpha\neq0$. Then, the endomorphism $J$ defined by $J|_{\mathfrak{v}}= \frac{1}{\sqrt{-\lambda}}\alpha_2$ and $J|_{\mathfrak{z}}= \frac{1}{3\sqrt{-\lambda}}\alpha_0$ turns out to be a bi-invariant complex structure on $\mathfrak{n}$. 
 
Conversely,  let $J$ be an orthogonal bi-invariant complex structure on $\mathfrak{n}$. The
bi-invariance of  $J$ implies that it preserves the center and its ortogonal complement.  Defining  $\alpha_2:= J|_{\mathfrak{v}}$ and $\alpha_0:= 3J|_{\mathfrak{z}}$, $\alpha$ turns out to be  a Killing 2-form. 
\end{proof}

Note that, in the argument of Andrada and Dotti, if we assume that $\n$ is irreducible, then $\mathfrak a=0$, $r=1$ in the notation above, and the endomorphism $T$ (and $J$) can be written as
$$T=\begin{pmatrix}
	0&-3aI_p&&\\
	3aI_p&0&&\\
	&&0&-aI_q\\
	&&aI_q&0
\end{pmatrix},$$
where $p=\dim\z$ and $q=\dim\v$, and therefore  there is only one KY tensor, up to scaling.
This result was generalized by del Barco and Moroianu. Indeed, exploiting the de Rham decomposition of $(N,g)$ as in Theorem \ref{deRham}, they showed their main result concerning to KY $2$-forms.

\begin{teo}\cite[Theorem 4.11]{dBM19}
	Let $(N, g)$ be a simply connected $2$-step nilpotent Lie group endowed
	with a left-invariant Riemannian metric. Then any invariant Killing 2-form is the sum
	of left-invariant Killing 2-forms on its de Rham factors. Moreover, the dimension of 
	the space of left-invariant Killing 2-forms on $(N, g)$, is $$\dim\mathcal K^2(N,g)=\frac{d(d-1)}{2}+r$$
	where $d$ is the dimension of the Euclidean factor in the de Rham decomposition of
	$(N, g)$, and $r$ is the number of irreducible de Rham factors admitting bi-invariant
	orthogonal complex structures.
\end{teo}

As a consequence of this result a classification in low dimensions is obtained.

\begin{teo}\cite[Theorem 4.14]{dBM19}
	There exist exactly 14 isomorphism classes of (non-abelian) $2$-step
	nilpotent Lie algebras of dimension $p \leq 8$ admitting an inner product for which
	the corresponding simply connected Riemannian Lie group carries non-zero KY
	2-forms: $\R^2\oplus\h_3$,  $\R^3\oplus\h_3$, $\h_3^\C$, $\R\oplus\h_3^\C$, $\R^2\oplus\h$ with $\h\in\mathcal{N}_5$ or $\h\in\mathcal{N}_6$.	
\end{teo}
In the above theorem $\mathcal{N}_i$ with $i=5,6$ denotes the set of isomorphism classes of real  $2$-step nilpotent Lie algebras of real dimension $i$. Note that $\dim\mathcal{N}_5=3$ and $\dim\mathcal{N}_6=7$.

\section{Higher degree Killing forms on $2$-step nilpotent Lie algebras}

We discuss here KY $p$-forms on $2$-step nilpotent Lie algebras. We split into two cases, namely $p=3$ and then $p>3$.
The case $p=3$ was considered in \cite{dBM19}, and its approach is the same as for $2$-forms.
They first reduced the study of  KY $3$-forms on $2$-step nilpotent Lie algebras to the case where $\n$ has no abelian factor. Indeed, \cite[Proposition 5.3]{dBM19} shows that a KY $3$-form on $\n$ is a sum of a Killing form on $\n_0$ and a $3$-parallel form on $\mathfrak a$.
Then, they  considered only irreducible Lie algebras, since \cite[Proposition 5.6]{dBM19} shows that a KY $3$-form on $\n$ is a sum of Killing forms on each irreducible ideal. After that, the authors obtained the following characterization of irreducible $2$-step nilpotent Lie algebras admitting KY $3$-forms. Recall the definition of the skew-symmetric endomorphism $j(z) : \v \to \v$ associated to a $2$-step nilpotent metric Lie algebras given in \eqref{jz}.

\begin{prop}\label{3-form_conditions}\cite[Proposition 5.8]{dBM19}
An irreducible $2$-step nilpotent Lie algebra $\mathfrak{n}$ admits a non-zero Killing
3-form if and only if the following conditions hold:
\begin{enumerate}
	\item $j(\z)$ is a subalgebra of $\mathfrak{so}(\v)$,
	\item for each $z\in\z$, the map $z'\to j^{-1}[j(z),j(z')]$ is in $\mathfrak{so}(\z)$.
\end{enumerate}
In this case the space of KY $3$-forms is $1$-dimensional.
\end{prop}

The latter two conditions  are related to the property of the
corresponding simply connected Riemannian Lie group to be naturally reductive. 
Recall that a homogeneous Riemannian manifold $(M, g)$ is naturally reductive if there is a
transitive group $G \subset Iso(M, g)$ and a reductive decomposition
$\g = \h \oplus\mathfrak m$ of the Lie algebra $\g$ of $G$, such that
$g([x, y]_{\mathfrak m}, z) + g(y, [x, z]_{\mathfrak m}) = 0$, for all $x, y, z \in\mathfrak m$.

The next result characterizes $2$-step nilpotent
Lie groups which are naturally reductive in terms of the conditions in the last proposition. For more details of this relation see \cite{Gordon}.

\begin{teo}\cite[Theorem 5.10]{dBM19}
Let $(N, g)$ be a simply connected $2$-step nilpotent Lie group
without Euclidean factor (in the de Rham sense). Let $\n$ denotes its Lie algebra and consider the orthogonal
decomposition $\n = \z\oplus\v$. Then $N$ is naturally reductive if and only if $\n$ satisfies the
conditions in Proposition \ref{3-form_conditions}.
\end{teo}

With all these ingredients we can state the main result about KY $3$-forms on $2$-step nilpotent Lie group.

\begin{teo}\cite[Theorem 5.11]{dBM19}
	Let $(N, g)$ be a simply connected $2$-step nilpotent Lie group endowed
	with a left-invariant Riemannian metric. Then any invariant Killing 3-form is the sum
	of left-invariant Killing 3-forms on its de Rham factors. Moreover, the dimension of
	 the space of left-invariant Killing 3-forms on $(N, g)$, is
	 $$\dim\mathcal K^3(N,g)=\frac{d(d-1)(d-2)}{6}+r$$
	 where $d$ and $r$ denote the dimension of the Euclidean factor and the number of naturally
	 reductive factors in the de Rham decomposition of $(N, g)$, respectively.
\end{teo}

Using the relation between $2$-step nilpotent Lie algebras of naturally reductive
type and representations of compact Lie algebras shown in \cite{Lauret}, the authors concluded with a classification of $2$-step nilpotent Lie algebras of naturally reductive type of dimension $\leq 6$.

\begin{teo}
	There exist exactly 8 isomorphism classes of (non-abelian) $2$-step nilpotent
	Lie algebras of dimension $\leq 6$ admitting an inner product for which the
	corresponding simply connected Riemannian Lie group carries non-zero Killing $3$-forms: 
	$\h_3$,  $\R\oplus\h_3$, $\R^2\oplus\h_3$, $\h_5$, $\R^3\oplus\h_3$, $\h_3\oplus\h_3$, $\R\oplus\h_5$, $\n_{3,2}$ (the free $2$-step nilpotent Lie algebra in dimension $5$).	
\end{teo}

\begin{rem}
	By comparing the results for KY $2$ and $3$-forms on $2$-step nilpotent 
	Lie groups $\n$ we have that a $2$-step nilpotent Lie group $N$ endowed with a left-invariant metric
	$g$ which is de Rham irreducible cannot admit non-zero Killing 2-forms and non-zero
	Killing 3-forms simultaneously \cite[Proposition 5.15]{dBM19}. In other words, a naturally reductive $2$-step nilpotent Lie group endowed with a left-invariant
	metric does not admit orthogonal bi-invariant complex structures \cite[Corollary 5.16]{dBM19}.
\end{rem}

\medskip

Let us focus now on the case $p>3$. In \cite{dBM20}, the authors investigated left-invariant Killing $p$-forms of arbitrary degree on
simply connected $2$-step nilpotent Lie groups endowed with left-invariant Riemannian metrics. They also classified them when the center has dimension at most $2$.

Let us to consider first the case of a non-abelian $2$-step nilpotent Lie algebra $\n$ with a center of dimension one. It is easy to prove that in this situation, $\n$ is isomorphic to the Heisenberg Lie algebra, and the metric is encoded in a matrix $A$ such that $A :=j(\xi) \in\mathfrak{so}(\v)$, so that $j(z) =g(z, \xi)A$ for every $z\in\z$ where $\xi$ is a  unit vector in $\z$. 
The main result for this case is:

\begin{teo}\cite[Theorem 5.3]{dBM20}
	The space $\mathcal K^k(\n, \pint)$ of Killing k-forms on a $2$-step nilpotent metric Lie algebra $(\n, \pint)$ with $1$-dimensional center is zero for $k$ even and is $1$-dimensional and generated by $\xi\wedge A \wedge\dots\wedge A$ for any $k$ odd with $k\leq\dim(\n)$, where $A$ denotes the $2$-form associated to the skew-symmetric endomorphism $A$.
\end{teo}	

\

Consider now the case when the center of the $2$-step nilpotent Lie algebra $\n$ $2$-dimensional. Let ${z_1, z_2}$ be an orthonormal basis of $\z$ and denote by $A_i:=j(z_i) \in\mathfrak{so}(\v)$, for $i=1,2$. The main result here states that:

\begin{teo}\cite[Theorem 6.2]{dBM20}
	The space $\mathcal K^k(\n, \pint)$ of Killing k-forms on a $2$-step nilpotent metric Lie algebra $(\n, \pint)$ with 2-dimensional center is zero for $4 \leq k\leq \dim(\n)-1$.
\end{teo}	

For KY $2$- and $3$-forms on a $2$-step nilpotent Lie algebra the authors obtained the following result:
\begin{teo}\cite[Theorem 6.6]{dBM20}
	Let $(\n, \pint)$ be a $2$-step nilpotent metric Lie algebra with $2$-dimensional center $\z$. If $(\n, \pint)$ is irreducible, then the space of Killing forms on $(\n, \pint)$ satisfies:
	\begin{itemize}
		\item $\mathcal K^1(\n, \pint) =\z$ is $2$-dimensional.
		\item $\mathcal K^2(\n, \pint)$ is 1-dimensional if $\n$ admit a bi-invariant $\pint$-orthogonal complex structure, and zero otherwise.
		\item $\mathcal K^3(\n, \pint)$ is 1-dimensional if $(\n,\pint)$ is naturally reductive, and zero otherwise.	
		\item $\mathcal K^k(\n, \pint)=0$ for $4 \leq k\leq \dim(\n)-1$.
		\item $\mathcal K^k(\n, \pint)=\Lambda^k\n$ is 1-dimensional for $k=\dim(\n)$.
	\end{itemize}
If $(\n, \pint)=(\n_1, \pint_1)\oplus(\n_2, \pint_2)$ is reducible, with $\dim\n_1\leq\dim\n_2$, then the space of KY forms on $(\n, \pint)$ satisfies:
	\begin{itemize}
	\item $\mathcal K^1(\n, \pint) =\z$ is $2$-dimensional.
	\item $\mathcal K^k(\n, \pint)=0$ if $k$ is even and  $2 \leq k\leq \dim(\n)$.
	\item If $k$ is odd and $3\leq k\leq \dim(\n)-1$, then $\mathcal K^k(\n, \pint)$ is 2-dimensional for $k\leq\dim\n_1$, 1-dimensional for $\dim\n_1< k\leq\dim\n_2$ and zero for $k>\dim\n_2$.
	\item $\mathcal K^k(\n, \pint)=\Lambda^k\n$ is 1-dimensional for $k=\dim(\n)$.
\end{itemize}
\end{teo}	

\medskip

\section{Strict CKY forms on $2$-step nilpotent Lie algebras}

In this section, CKY $2$-forms on $2$-step nilpotent Lie algebras are described. The first result shows that the Heisenberg Lie algebra is the only $2$-step nilpotent one admitting strict CKY $2$-forms. It was proved  first in \cite[Theorem 4.1]{ABD}.

\begin{teo}\cite[Theorem 4.1]{ABD}
A $2$-step nilpotent Lie algebra $\mathfrak{g}$ admitting an inner product with a CKY tensor $T$ which is not KY is isomorphic to $\mathfrak{h}_{2n+1}$, and if $\xi$ generates the center of $\mathfrak{h}_{2n+1}$, then $T|_{\z}=0$ and $T|_{\mathfrak{v}}=\lambda j_{\xi}^{-1}$, for some $\lambda\neq 0$ and $\mathfrak{v}=\xi^{\perp}$
\end{teo}
Moreover, it was showed in Corollary \ref{KY en heisenberg} that there are no non-trivial KY tensors on $\h_{2n+1}$.
Recently, in  \cite{dBM21} the authors obtained the following similar result.
\begin{teo}\cite[Theorem 4.1]{dBM21}
 Every CKY $2$-form on a $2$-step nilpotent metric Lie algebra $(\mathfrak g,\pint)$ is a Killing form, except when $\mathfrak g$ is the Heisenberg Lie algebra, for every metric $\pint$. In this case  $\mathcal{ CK}^2(\mathfrak g,\pint)$ is $1$-dimensional and $\mathcal{K}^2(\mathfrak g,\pint)=0$
\end{teo}

\begin{rem}
For any $2$-step nilpotent Lie algebra $\mathcal{ CK}^2(\mathfrak g,\pint)=\mathcal{ K}^2(\mathfrak g,\pint)$ except when $\mathfrak g$ is the Heisenberg Lie algebra $\h_{2n+1}$, and in this case  $\mathcal{ CK}^2(\mathfrak h_{2n+1},\pint)=\mathcal{ *K}^2(\mathfrak h_{2n+1},\pint)$.	Such a description of conformal Killing $2$-forms does not hold on arbitrary metric Lie algebras. Indeed, in \cite{HO} we show examples of metric Lie algebras carriying CKY $2$-forms which are not a linear combination of a KY and a $*$-KY $2$-forms (see the $3$-step nilpotent Lie algebra $L_{5,9}$ in Table \ref{tabla_dim5_xi_no_central}).
\end{rem}

\

The case of CKY $3$-forms is more involved, it depends on the $\mathfrak v$-degree of such $3$-form (see \eqref{v-degree}). Before stating the main result, we need to introduce some notation.
Let $\n_{3,2}$ be the 6-dimensional $2$-step nilpotent Lie algebra with Lie bracket given by
\begin{equation*}
	[x_1, x_2] = y_3, \, [x_2, x_3] = y_1,\, [x_3, x_1] = y_2.
\end{equation*}
Now, we consider the 1-parameter family of metrics $\pint_\lambda$, with $\lambda\in\R_{>0}$, on $\n_{3,2}$ such that $\{ x_1, x_2,x_3,\frac{y_1}{\lambda},\frac{y_2}{\lambda},\frac{y_3}{\lambda}\}$ is a $\pint_\lambda$-orthonormal basis.

\begin{teo}\cite[Theorem 5.6]{dBM21}
	Every CKY $3$-form on a $2$-step nilpotent metric Lie algebra $(\mathfrak g,\pint)$ is a Killing form, except when $\mathfrak g=\h_{2n+1}\times\R$ and $\pint$ is any metric, and when $(\mathfrak g,\pint)$ is $(\mathfrak n_{3,2},\pint_\lambda)$ for some $\lambda\in\R_{>0}$. In both cases, $\mathcal{ CK}^3(\mathfrak g,\pint)$ is $2$-dimensional and $\mathcal{K}^3(\mathfrak g,\pint)$ is $1$-dimensional.
\end{teo}

\begin{rem}
	As a consequence of the proof of the last result, it can be shown that for any $2$-step nilpotent metric Lie algebra $\mathcal{ CK}^3(\mathfrak g,\pint)=\mathcal{ K}^3(\mathfrak g,\pint)$ except when $\mathfrak g=\h_{2n+1}\times\R$ and $\pint$ is any metric, and when $(\mathfrak g,\pint)$ is $(\mathfrak n_{3,2},\pint_\lambda)$ where  
	$$\mathcal{ CK}^3(\mathfrak h_{2n+1}\times\R,\pint)=\mathcal{CK}_{ev}^3(\mathfrak h_{2n+1}\times\R,\pint)=\mathcal{K}_{ev}^3(\mathfrak h_{2n+1}\times\R,\pint)\oplus\mathcal{*K}_{ev}^3(\mathfrak h_{2n+1}\times\R,\pint)$$	
	$$\mathcal{ CK}^3(\mathfrak n_{3,2},\pint)=\mathcal{CK}_{ev}^3(\mathfrak n_{3,2},\pint)\oplus\mathcal{CK}_{odd}^3(\mathfrak n_{3,2},\pint)=\mathcal{K}_{ev}^3(\mathfrak n_{3,2},\pint)\oplus\mathcal{*K}_{ev}^3(\mathfrak n_{3,2},\pint)$$	
\end{rem}

\medskip

\section{Open problems}
In this section we will discuss interesting open problems on left-invariant CKY forms.

\subsection{CKY forms on 5-dimensional metric Lie algebras}

According to \cite[Corollary 4.4]{AD} CKY $2$-forms which are not KY only occur on odd-dimensional Lie algebras. Since the $3$-dimensional case was classified in \cite{ABD}, then the next natural step is dimension $5$. As mentioned before, in \cite{HO} we started with this case. However, we only classified 5-dimensional metric Lie algebras admitting CKY $2$-forms in the cases when the vector $\xi\in \z$ and when $\xi \perp \z$.

\

\textbf{Open question:}
Determine the existence
of CKY 2-forms on an arbitrary $5$-dimensional metric Lie algebra with 1-dimensional center, which
is not necessarily generated by $\xi$. In particular, it would be interesting to study other
metrics on those 5-dimensional Lie algebras in Table \ref{NotacionNueva} and determine if they admit a CKY 2-form with $\xi\notin\z$. 

\

\textbf{Open question:}
Are there examples of 5-dimensional metric Lie algebras with trivial center admitting CKY $2$-forms?

After that we could have a complete picture about CKY tensors on $5$-dimension metric Lie algebras.

\

\textbf{Open question:} Study CKY $p$-forms on $5$-dimensional metric Lie algebras. Using the description of CKY $2$-forms, it would be possible to analyze CKY $3$-forms by talking its $*$ dual operator. After that it remains to analyze CKY $4$-forms or equivalently, CKY vector fields.

\

\subsection{$3$-step nilpotent metric Lie algebras}
As far as we know, the only example of a $3$-step nilpotent Lie algebra carrying a CKY 2-form is $(L_{5,9},\pint_r)$ in Table \ref{tabla_dim5_xi_no_central}, the only $3$-step nilpotent Lie algebra in dimension $5$.  Moreover, the CKY 2-form is not a linear combination of a KY $2$-form and a $*$-dual of a KY $2$-form. Therefore, it would be very interesting to study its properties and behaviour.

\

\textbf{Open question:} Is it possible to generalize this example to higher dimension?

\

\subsection{Construction of CKY 2-forms}
The main tool to construct examples of metric Lie algebras admitting strict CKY $2$-forms is Theorem \ref{teorema dotti-andrada}. That result uses an invertible KY tensor $S$ such that the  $2$-form $\mu(x,y)=-2\la S^{-1}x,y\ra$ is closed to produce a CKY tensor.
So far, all known examples of tensors $S$ satisfying these condition are parallel. Then, the next natural question arise.

\

\textbf{Open question:} Let $S$ be an invertible KY tensor such that $\mu(x,y)=-2\la S^{-1}x,y\ra$ is closed. Is $S$ parallel?

\

\subsection{The compact case}
We know from Theorem \ref{bi inv dim 3} that the existence of CKY tensors on compact Lie groups equipped with a bi-invariant metric is only possible in dimension $3$, and indeed the Lie algebra is isomorphic to $\mathfrak{su}(2)$.

\

\textbf{Open question:} Do there exist examples of compact Lie groups admitting a CKY tensor with general left-invariant metric?

\

\subsection{The almost abelian case}
As we mentioned before, CKY $2$-forms on almost abelian Lie groups were studied in \cite{AD}.
They proved that only in dimension $3$ is possible to find examples of strict CKY $2$-forms.
Moreover, if it is a KY $2$-form, then it is parallel, see Theorem \ref{almost-abelian-2-forms}.

\

\textbf{Open question:} Are there invariant CKY $p$-forms with $p>2$ on almost abelian Lie groups?

\

\subsection{Homogeneous spaces}
More generally, one can consider homogeneous spaces $G/K$ endowed with a $G$-invariant metric. It is possible to translate the definition of CKY forms to $(G/K,g)$ where $g$ is a $G$-invariant metric. In this context, study $G$-invariant CKY $2$-forms on $(G/K,g)$ is equivalent to study $G$-invariant and skew-symmetric  $(1,1)$-tensors $H : T (G/K) \to T (G/K)$ such that
\begin{equation*}
	\label{cky3}
	g\left(\left(\nabla_ X H\right)Y, Z\right) + g\left(\left(\nabla_ Y H\right)X, Z\right) = 2g(X, Y )\theta(Z)- g(X, Z)\theta(Y ) -g(Y, Z)\theta(X).
\end{equation*}
In \cite{Dotti-Herrera} full flag manifolds $SU(n)/ T$ for every $n\geq 4$ are analyzed. Despite this result, and in contrast with the results we mentioned of CKY $p$-forms on Lie groups, there are not many results of CKY $p$-forms on homogeneous spaces with $G$-invariant metrics. 

\

\textbf{Open question:} Look for examples of CKY $p$-forms on homogeneous spaces. In particular, it would be interesting  to analyze the existence of CKY $p$-forms on  flags manifolds and Grasmannian manifolds.

\

\

\noindent \textbf{Acknowledgements.} We are grateful to the organization of the problem sessions during the Workshop ``Geometric structures and Moduli spaces'' in C\'ordoba where this project emerged. We also would like to thank Adri\'an Andrada for his advises and interesting suggestions, and the anonymous referees for their comments. Both authors were partially supported by CONICET, ANPCyT and SECyT-UNC 
(Argentina).

\

\end{document}